\documentclass[11pt]{article}

\usepackage{amsmath,amsfonts,amssymb,amsthm,mathtools}
\usepackage[top=2.2cm,bottom=2.2cm,left=2.6cm,right=2.6cm]{geometry}
\usepackage{xcolor}
\usepackage{hyperref}
\hypersetup{
  colorlinks=true,
  citecolor=blue,
  linkcolor=blue,
  urlcolor=blue
}

\newtheorem{theorem}{Theorem}[section]
\newtheorem{lemma}[theorem]{Lemma}

\newtheorem{corollary}[theorem]{Corollary}

\theoremstyle{definition}

\theoremstyle{remark}

\newcommand{\Hh}{\mathcal H}
\newcommand{\Ff}{\mathcal F}
\newcommand{\Ll}{\mathcal L}
\newcommand{\Ss}{\mathcal S}
\newcommand{\Tt}{\mathcal T}
\newcommand{\Dd}{\mathcal D}
\newcommand{\supp}{\operatorname{supp}}
\newcommand{\spex}{\operatorname{spex}}

\renewenvironment{abstract}
  {\par\noindent\textbf{\abstractname.}\ \ignorespaces}
  {\par\medskip}
\providecommand{\keywords}[1]{\small\textbf{\textit{Keywords.}} #1}
\providecommand{\ams}[1]{\small\textbf{\textit{AMS subject classifications.}} #1}

\title{\bfseries\Large Spectral extremal hypergraphs without long Berge cycles}

\author{
Lihua Feng$^a$,\ \  Lu Lu$^a$\footnote{Corresponding author. \newline{\hspace*{5mm} Email addresses}:
\url{fenglh@163.com} (L. Feng), \url{lulumath@csu.edu.cn} (L. Lu), \url{mathtzwu@163.com} (T. Wu).}, \ \ Tingzeng Wu$^b$\\[2mm]
\small $^a$School of Mathematics and Statistics, HNP-LAMA, Central South University,\\
\small $^a$Changsha 410083, China\\
\small $^b$School of Mathematics and Statistics, Qinghai Minzu University,\\
\small $^b$Xining, Qinghai, 810007, China\\}

\date{}

\begin{document}
\maketitle

\begin{abstract}
Let $r\ge 3$ and $k\ge 2r+1$ be fixed integers. We determine, for all sufficiently
large $n$, the maximum adjacency-tensor spectral radius of an $n$-vertex
$r$-uniform hypergraph containing no Berge cycle of length at least $k$.
Write $s=\left\lfloor\frac{k-1}{2}\right\rfloor$. If $k=2s+1$ is odd, the unique extremal hypergraph consists of all
$r$-sets containing at most one vertex outside a fixed $s$-set. If
$k=2s+2$ is even, one additionally includes all $r$-sets containing a
fixed pair outside the $s$-set together with $r-2$ vertices inside it.
Consequently, the maximum spectral radius is given as
\[
\spex_r(n,k)
=
\left[
\binom{s}{r-1}
\left(\frac{r-1}{s}\right)^{(r-1)/r}
+o(1)
\right]n^{(r-1)/r}.
\]
\end{abstract}

\noindent\keywords{Uniform hypergraph; Berge cycle; Spectral extremal problem; Adjacency tensor}

\smallskip
\noindent\ams{05C50, 05D05, 05C65}

\section{Introduction}

An \emph{$r$-uniform hypergraph}, or simply an \emph{$r$-graph}, is a family of
$r$-element subsets of a finite vertex set. For an $r$-graph $\Hh$, we
write $V(\Hh)$ and $E(\Hh)$ for its vertex and edge sets, respectively,
and put $e(\Hh)=|E(\Hh)|$. A \emph{Berge cycle of length $\ell$} consists
of distinct vertices $v_1,\ldots,v_\ell$ and distinct edges
$e_1,\ldots,e_\ell$ such that
$\{v_i,v_{i+1}\}\subseteq e_i$ for $1\le i\le \ell$, where indices are taken modulo $\ell$. Thus a Berge cycle is obtained
from an ordinary cycle by enlarging its edges to distinct hyperedges.
We say that $\Hh$ is \emph{Berge-$C_{\ge k}$-free} if it contains no
Berge cycle of length at least $k$.

The classical theorem of Erd\H{o}s and Gallai \cite{ErdosGallai1959}
states that an $n$-vertex graph with no cycle of length at least $k$ has
at most $(k-1)(n-1)/2$ edges. Hypergraph analogues were first developed
for Berge paths, notably by Gy\H{o}ri, Katona and Lemons
\cite{GyoriKatonaLemons2016}. For Berge cycles, F\"uredi, Kostochka and
Luo \cite{FKL1} proved that if $r\ge 3$, $k\ge r+3$, and $\Hh$ is an
$n$-vertex Berge-$C_{\ge k}$-free $r$-graph, then
\begin{equation}\label{eq:1}
 e(\Hh)\le \frac{n-1}{k-2}\binom{k-1}{r}.
\end{equation}
They also characterized equality when $(k-2)\mid(n-1)$. Their sequel
\cite{FKL2} determines the exact edge extremal function for every $n$
when $k\ge r+4$ and describes all extremal configurations. The cases
$k=r+1$ and $k=r+2$ were settled by Ergemlidze, Gy\H{o}ri, Methuku,
Salia, Tompkins and Zamora \cite{ErgemlidzeEtAl2020}.

Tensor methods provide a natural spectral theory for uniform
hypergraphs. Eigenvalues of symmetric tensors were introduced by Qi
\cite{Qi2005} and Lim \cite{Lim2005}, and Perron--Frobenius theory for
nonnegative tensors was developed by Chang, Pearson and Zhang
\cite{ChangPearsonZhang2008}. Cooper and Dutle \cite{CooperDutle2012}
associated an adjacency tensor with a uniform hypergraph and developed
many analogues of basic results from spectral graph theory. For a
nonnegative vector $x=(x_v)_{v\in V(\Hh)}$, write
\[
P_{\Hh}(x)
:=
r\sum_{e\in E(\Hh)}\prod_{v\in e}x_v.
\]
The adjacency-tensor spectral radius of $\Hh$ is
\begin{equation}\label{eq:2}
\rho(\Hh)
:=
\max\left\{
P_{\Hh}(x):x_v\ge 0,\ \sum_{v\in V(\Hh)}x_v^r=1
\right\}.
\end{equation}

Spectral extremal problems for uniform hypergraphs have been studied
from several directions. Keevash, Lenz and Mubayi
\cite{KeevashLenzMubayi2014} established general principles that
transfer strong stability results from ordinary Tur\'an problems to
$\alpha$-spectral extremal problems. Nikiforov \cite{Nikiforov2014}
developed a broad analytic theory of the $p$-spectral radius, while
Kang, Nikiforov and Yuan \cite{KangNikiforovYuan2015} solved several
$p$-spectral extremal problems for multipartite and chromatic uniform
hypergraphs. The weighted-incidence method of Lu and Man
\cite{LuMan2016}, and its $p$-spectral extension by Liu and Lu
\cite{LiuLu2022}, supplies another useful comparison tool. In the sparse
regime, Bai and Lu \cite{BaiLu2018} proved a hypergraph analogue of
Stanley's fixed-edge bound: if $e(\Hh)=\binom{x}{r}$ for a real
$x\ge r-1$, then
\[
\rho(\Hh)\le \binom{x-1}{r-1},
\]
where equality holds precisely for a complete $r$-graph together with isolated
vertices when $x$ is integral. Recent stability developments include
results for cancellative hypergraphs \cite{NiLiuKang2024} and for broad
classes of nondegenerate Tur\'an problems \cite{ZhengLiFan2026}.

There is also a growing literature on spectral problems with Berge-type
forbidden configurations. Zhou, Kang, Liu and Shan
\cite{ZhouKangLiuShan2020} and Kang, Liu, Lu and Wang
\cite{KangLiuLuWang2021} studied $p$-spectral extremal problems for
Berge hypergraphs. Fang, Chang, Xu, Gao and Hou
\cite{FangChangXuGaoHou2025} recently determined the exact spectral
extremizer for sufficiently large Berge-path-free $3$-graphs in a broad
range. For linear uniform hypergraphs, She, Fan and Kang
\cite{SheFanKang2025} obtained spectral Tur\'an bounds for several
bipartite Berge patterns. Wang, Huang and Shi
\cite{WangHuangShi2026} considered linear hypergraphs avoiding
expansions of paths and gave a spectral bound for connected linear
$3$-graphs without a prescribed Berge path. At the Hamiltonian end,
Brooks, Linz and Luo \cite{BrooksLinzLuo2026} obtained an
adjacency-spectral condition forcing a Hamiltonian Berge cycle, apart
from an explicit extremal exception.

In this paper we maximize the spectral radius over the entire family of
$r$-graphs with no Berge cycle of length at least $k$. For fixed $r$ and
$k$, define
\[
\spex_r(n,k)
:=
\max\left\{
\rho(\Hh):|V(\Hh)|=n,\ \Hh\text{ is an $r$-graph and is Berge-$C_{\ge k}$-free}
\right\}.
\]
Throughout the paper, we put
\[
q:=r-1,
\qquad
s:=\left\lfloor\frac{k-1}{2}\right\rfloor,
\qquad
L_{r,s}:=\binom{s}{q}\left(\frac{q}{s}\right)^{q/r}.
\]
The assumption $k\ge 2r+1$ is equivalent to $s\ge r$.

We next describe the two extremal constructions. Let $S$ be an
$s$-element set, put $Y=V\setminus S$, and define
\begin{equation}\label{eq:3}
E(\Ss_{n,s}^{(r)})
=
\binom{S}{r}
\cup
\left\{
\{y\}\cup F:y\in Y,\ F\in\binom{S}{q}
\right\}.
\end{equation}
Thus $\Ss_{n,s}^{(r)}$ consists of all $r$-sets containing at most one
vertex outside $S$. For the even construction, assume \(n\ge s+2\), choose distinct
vertices \(u,v\in Y\), and put
\begin{equation}\label{eq:4}
E(\Ss_{n,s}^{(r),+})
=
E(\Ss_{n,s}^{(r)})
\cup
\left\{
\{u,v\}\cup R:R\in\binom{S}{r-2}
\right\}.
\end{equation}
The isomorphism type is independent of the chosen pair $\{u,v\}$.

Our main result determines the exact extremal hypergraphs for all
sufficiently large $n$.

\begin{theorem}\label{thm:main}
Fix $r\ge 3$ and $k\ge 2r+1$, and let
$s=\lfloor(k-1)/2\rfloor$. There exists $n_0=n_0(r,k)$ such that, for
all $n\ge n_0$, the following statements hold.
\begin{itemize}
\item[(i)] If $k=2s+1$, then every $n$-vertex Berge-$C_{\ge k}$-free
$r$-graph $\Hh$ satisfies
\[
\rho(\Hh)\le \rho(\Ss_{n,s}^{(r)}),
\]
with equality if and only if $\Hh\cong\Ss_{n,s}^{(r)}$.
\item[(ii)] If $k=2s+2$, then every $n$-vertex Berge-$C_{\ge k}$-free
$r$-graph $\Hh$ satisfies
\[
\rho(\Hh)\le \rho(\Ss_{n,s}^{(r),+}),
\]
with equality if and only if $\Hh\cong\Ss_{n,s}^{(r),+}$.
\end{itemize}
\end{theorem}

\begin{corollary}\label{cor:asymptotic-intro}
Under the assumptions of Theorem~\ref{thm:main}, we have
\[
\spex_r(n,k)
=
\left(L_{r,s}+o(1)\right)n^{q/r}.
\]
\end{corollary}

The spectral extremal configuration is genuinely different from the
edge-extremal one. When \((k-2)\mid(n-1)\), equality in
\eqref{eq:1} is attained by a connected block structure whose blocks
are copies of \(K_{k-1}^{(r)}\) \cite{FKL1}; the exact edge extremizers
for arbitrary \(n\) retain this block-based character
\cite{FKL2}. By contrast, the spectral extremizer concentrates its
large Perron coordinates on a fixed set \(S\) of \(s\) vertices and
attaches linearly many low-weight vertices through the complete
\((r-1)\)-link over \(S\). When \(k\) is even, one additional complete
family supported on a single pair outside \(S\) is permitted.

For comparison, Gao and Hou \cite{GaoHou} solved the corresponding
adjacency-spectral problem for ordinary graphs. For two vertex-disjoint
graphs \(G_1\) and \(G_2\), let \(G_1\vee G_2\) denote their join, and
let \(\overline{K}_t\) denote the empty graph on \(t\) vertices. For
all sufficiently large \(n\), their unique extremal graph is $K_s\vee\overline{K}_{n-s}$ when $k=2s+1$, and
$K_s\vee\left(K_2\cup\overline{K}_{n-s-2}\right)$ when \(k=2s+2\). These graphs are precisely the \(2\)-shadows of
\(\mathcal{S}_{n,s}^{(r)}\) and
\(\mathcal{S}_{n,s}^{(r),+}\), respectively.

More recently, Wang, Sun and Bu \cite{WangSunBu} obtained the
analogous conclusion for the \(t\)-clique spectral radius of graphs.
Since the \(t\)-clique tensor of a graph is the adjacency tensor of
its \(t\)-uniform clique hypergraph, their result concerns a special
class of uniform hypergraphs. The present theorem treats arbitrary
\(r\)-graphs, for which there need not be an underlying graph
controlling all hyperedges, and the requirement that the hyperedges
representing a Berge cycle be pairwise distinct becomes an additional
difficulty.

The proof proceeds in three steps. We first apply the hereditary edge
bound \eqref{eq:1} to assign each edge to one of its vertices in such a
way that the number of edges assigned to any given vertex is bounded
by a constant depending only on $r$ and $k$. Let $x_1\ge x_2\ge\cdots\ge x_n$ be the entries of a normalized positive eigenvector corresponding to
the spectral radius. The preceding assignment implies that, after a
fixed number of the largest entries have been selected, the total
contribution of edges containing at least two of the remaining
vertices is of lower order. A vertex-replacement argument then
restricts the collections of $(r-1)$-subsets that can extend many of
the remaining vertices to edges. Consequently, the leading term is
determined by the maximization of an explicit function of the total
$r$-mass carried by the largest entries, which yields the asymptotic
formula.

We next establish a quantitative stability result for the associated
eigenvector. If the spectral radius is sufficiently close to the
asymptotic maximum, then the first $s$ entries are close to one
another, the sum of their $r$th powers is close to
$(r-1)/r$, and there exists a set $S$ consisting of the
corresponding vertices such that, for all but $o(n)$ vertices
$v\notin S$, $A\cup{v}\in E(\mathcal H)$ for every $A\in\binom{S}{r-1}$.
Finally, by comparing the weighted contributions of the edges present
in $\mathcal H$ with those of the corresponding missing edges in
the proposed extremal construction, we show that every vertex outside
$S$ has this property. The odd and even cases are treated
simultaneously up to the final structural step. When $k$ is odd, no
edge can contain two vertices outside $S$. When $k$ is even, every
edge containing at least two vertices outside $S$ must contain the
same pair of such vertices.

\section{Preliminaries}

For an $r$-graph $\Hh$, its $2$-shadow is
\[
\partial_2\Hh
:=
\left\{
\{x,y\}:\{x,y\}\subseteq e\text{ for some }e\in E(\Hh)
\right\}.
\]
We call $\Hh$ connected if $\partial_2\Hh$ is connected. Every Berge
cycle in $\Hh$ projects to an ordinary cycle on the same base vertices
in $\partial_2\Hh$.

For a vertex $v$ and a set $U\subseteq V(\Hh)\setminus\{v\}$, define
the restricted link
\[
L_U(v)
:=
\left\{
F\in\binom{U}{q}:F\cup\{v\}\in E(\Hh)
\right\}.
\]
When $U=V(\Hh)\setminus\{v\}$, we simply write $L(v)$. For a vector
$x=(x_v)$ and a set $A$ of vertices, write
\[
x_A:=\prod_{v\in A}x_v.
\]
For nonnegative real numbers $z_1,\ldots,z_m$, let
\[
e_j(z_1,\ldots,z_m)
:=
\sum_{F\in\binom{[m]}{j}}\prod_{i\in F}z_i
\]
be the $j$th elementary symmetric polynomial.

If $\Hh$ is connected and has at least one edge, the tensor
Perron--Frobenius theorem
\cite{ChangPearsonZhang2008,CooperDutle2012} gives a positive vector
attaining the maximum in \eqref{eq:2}. We call
such a normalized vector a \emph{Perron vector}. It satisfies
\begin{equation}\label{eq:eigen-equation}
\rho(\Hh)x_v^{r-1}
=
\sum_{\substack{e\in E(\Hh)\\ v\in e}}x_{e\setminus\{v\}},
\end{equation}
for every $v\in V(\Hh)$.

\begin{lemma}[{\cite[Theorem 3.1]{CooperDutle2012}}]\label{lem:components}
If $\Hh_1,\ldots,\Hh_t$ are the connected components of an $r$-graph
$\Hh$, then
\[
\rho(\Hh)=\max_{1\le i\le t}\rho(\Hh_i).
\]
\end{lemma}

We next record the edge bound of F\"uredi, Kostochka and Luo.

\begin{theorem}[F\"uredi--Kostochka--Luo \cite{FKL1}]\label{thm:FKL}
Let $r\ge 3$, $k\ge r+3$, and $n\ge k$. If an $n$-vertex $r$-graph
$\Hh$ is Berge-$C_{\ge k}$-free, then $e(\Hh)\le \frac{n-1}{k-2}\binom{k-1}{r}$.
\end{theorem}

Put $\kappa:=\binom{k-1}{r}$. The following weaker hereditary form is all that we need.

\begin{lemma}\label{lem:sparsity}
Every Berge-$C_{\ge k}$-free $r$-graph $\Hh$ satisfies
\begin{equation}\label{eq:uniform-sparsity}
e(\Hh[U])\le \kappa |U|
\end{equation}
for every $U\subseteq V(\Hh)$.
\end{lemma}

\begin{proof}
If $|U|<k$, then
\[
e(\Hh[U])\le \binom{|U|}{r}\le \kappa\le \kappa|U|
\]
whenever $U\ne\varnothing$, while the empty case is trivial. If
$|U|\ge k$, then $\Hh[U]$ is again Berge-$C_{\ge k}$-free, and
Theorem~\ref{thm:FKL} gives
\[
e(\Hh[U])
\le
\frac{|U|-1}{k-2}\binom{k-1}{r}
\le
\kappa|U|.
\]
\end{proof}

The following assignment lemma is an immediate consequence of
Hall's marriage theorem.

\begin{lemma}\label{lem:bounded-assignment}
Let $\Ff$ be a finite family of sets, and let $B$ be a vertex set such
that every member of $\Ff$ meets $B$. Suppose that, for some positive
integer $K$,
\[
|\Ff'|
\le
K\left|B\cap\bigcup_{e\in\Ff'}e\right|
\]
for every subfamily $\Ff'\subseteq\Ff$. Then there exists a map $\psi:\Ff\to B$ such that
$\psi(e)\in e\cap B$ for every $e\in\Ff$ and $|\psi^{-1}(b)|\le K$ for every $b\in B$.
\end{lemma}

\begin{proof}
Construct an auxiliary bipartite graph $G$ with bipartition
$\Ff$ and $\widetilde B:=B\times\{1,\ldots,K\}$. Thus, the right part contains $K$ labelled copies
\[
(b,1),\ldots,(b,K)
\]
of each vertex $b\in B$. For $e\in\Ff$ and
$(b,j)\in\widetilde B$, join $e$ to $(b,j)$ if and only if
$b\in e\cap B$.

We verify Hall's condition for the left part $\Ff$. Let
$\Ff'\subseteq\Ff$ be arbitrary. A vertex $(b,j)\in\widetilde B$
belongs to the neighborhood $N_G(\Ff')$ precisely when
$b\in e\cap B$ for at least one $e\in\Ff'$. Therefore, we have 
\[
N_G(\Ff')
=
\left(
 B\cap\bigcup_{e\in\Ff'}e
\right)
\times\{1,\ldots,K\},
\]
and hence $|N_G(\Ff')|
=
K\left|B\cap\bigcup_{e\in\Ff'}e\right|$.
By the hypothesis, we have $|N_G(\Ff')|
=
K\left|B\cap\bigcup_{e\in\Ff'}e\right|
\ge |\Ff'|$.
Thus Hall's condition holds for every subfamily
$\Ff'\subseteq\Ff$. By Hall's marriage theorem, $G$ has a matching
$M$ that saturates every vertex of the left part $\Ff$.

For each $e\in\Ff$, let $(b,j)$ be the unique vertex of
$\widetilde B$ matched to $e$ by $M$, and define $\psi(e):=b$. Since $e$ is adjacent to $(b,j)$, the definition of $G$ gives
$b\in e\cap B$, so $\psi(e)\in e\cap B$. Finally, for any vertex $b\in B$, if $\psi(e)=b$, then $e$ is matched to one of
the $K$ copies $(b,1),\ldots,(b,K)$. Because $M$ is a matching, each of these copies is matched to at most
one member of $\Ff$. Consequently, at most $K$ members of $\Ff$ can
be mapped to $b$, and therefore $|\psi^{-1}(b)|\le K$.
This proves the lemma.
\end{proof}

\begin{lemma}\label{lem:orientation}
Let $\Hh$ be a Berge-$C_{\ge k}$-free $r$-graph. There exists a map
\[
\phi:E(\Hh)\to V(\Hh)
\]
such that $\phi(e)\in e$ for every $e\in E(\Hh)$ and $|\phi^{-1}(v)|\le \kappa$ for every $v\in V(\Hh)$.
\end{lemma}

\begin{proof}
For every subfamily \(\mathcal{F}'\subseteq E(H)\), put $U:=\bigcup_{e\in\mathcal{F}'}e$.
Since every edge of \(\mathcal{F}'\) is contained in \(U\), Lemma
\ref{lem:sparsity} gives $|\mathcal{F}'|
   \le e\bigl(H[U]\bigr)
   \le \kappa |U|$.
Thus the hypothesis of Lemma~\ref{lem:bounded-assignment} holds with
\(B=V(H)\) and \(K=\kappa\). Applying that lemma yields a map $\varphi:E(H)\longrightarrow V(H)$ such that \(\varphi(e)\in e\) for every \(e\in E(H)\) and
$|\varphi^{-1}(v)|\le \kappa$ for every \(v\in V(H)\), as required.
\end{proof}

We will also use the following sharp inequality for the core
coordinates.

\begin{lemma}\label{lem:elementary-symmetric}
Let $a_1,\ldots,a_s\ge 0$ and put $T=\sum_{i=1}^s a_i^r$. Then we have
\begin{equation}\label{eq:elementary-symmetric}
e_q(a_1,\ldots,a_s)
\le
\binom{s}{q}\left(\frac{T}{s}\right)^{q/r}.
\end{equation}
Equality holds if and only if $a_1=\cdots=a_s$.
\end{lemma}

\begin{proof}
By H\"older's inequality, we have
\[
e_q(a_1,\ldots,a_s)
\le
\binom{s}{q}^{1/r}
\left(
\sum_{F\in\binom{[s]}{q}}
\left(\prod_{i\in F}a_i\right)^{r/q}
\right)^{q/r}.
\]
Set $z_i=a_i^{r/q}$. Then $\sum_i z_i^q=T$. By Maclaurin's
inequality and the power-mean inequality, we obtain
\[
\frac{e_q(z_1,\ldots,z_s)}{\binom{s}{q}}
\le
\left(\frac{z_1+\cdots+z_s}{s}\right)^q
\le
\frac{z_1^q+\cdots+z_s^q}{s}
=
\frac{T}{s}.
\]
Combining the two estimates proves \eqref{eq:elementary-symmetric}.
Equality throughout is possible exactly when all $a_i$ are equal.
\end{proof}

For a uniform hypergraph $\Ll$, define $\supp(\Ll):=\bigcup_{e\in E(\Ll)}e$. For a new vertex $y$, its \emph{cone} over $\Ll$ is
\[
y*\Ll:=\{\{y\}\cup e:e\in E(\Ll)\}.
\]

\begin{lemma}\label{lem:clone}
Let $\Ll$ be a $q'$-graph with $q'\ge 2$, let
$t=|\supp(\Ll)|$, and assume $e(\Ll)\ge 2$. If
$y_1,\ldots,y_t$ are distinct vertices outside $\supp(\Ll)$, then $\bigcup_{i=1}^t(y_i*\Ll)$ contains a Berge cycle of length $2t$.
\end{lemma}

\begin{proof}
Put $X=\supp(\Ll)$. Define a graph $\Gamma$ on $X$ by joining distinct
vertices $x,x'$ whenever there are distinct edges $e,e'\in E(\Ll)$
with $x\in e$ and $x'\in e'$. For $e\in E(\Ll)$, let
\[
P_e:=\{x\in e:d_{\Ll}(x)=1\}.
\]
Two vertices of $X$ are nonadjacent in $\Gamma$ if and only if they
belong to the same set $P_e$ for some $e\in E(\Ll)$. Hence $\Gamma$ is
a complete multipartite graph whose nontrivial parts are the nonempty
sets $P_e$.

Fix $e\in E(\Ll)$ and write $p=|P_e|$. Choose
$e'\in E(\Ll)\setminus\{e\}$. Since $P_e\cap e'=\varnothing$, we obtain
\[
q'=|e'|\le t-p,
\]
while $p\le q'$. Thus $p\le t/2$. Every vertex of $\Gamma$ therefore
has degree at least $t/2$. Since $t\ge 3$, Dirac's theorem gives a
Hamilton cycle
\[
x_1x_2\cdots x_t x_1
\]
in $\Gamma$.

For each $i$, choose distinct $e_i,e_i'\in E(\Ll)$ with
$x_i\in e_i$ and $x_{i+1}\in e_i'$, where $x_{t+1}=x_1$. Represent
$x_i y_i$ by $\{y_i\}\cup e_i$ and $y_i x_{i+1}$ by
$\{y_i\}\cup e_i'$. The two edges used at a fixed $y_i$ are distinct,
and edges associated with different cone vertices are automatically
distinct. Hence
\[
x_1,y_1,x_2,y_2,\ldots,x_t,y_t
\]
is the base sequence of a Berge cycle of length $2t$.
\end{proof}

\begin{corollary}\label{cor:link-type}
Let $k\in\{2s+1,2s+2\}$, and let $\Hh$ be Berge-$C_{\ge k}$-free.
Fix a vertex set $X$ and a $q$-graph
$\Ll\subseteq\binom{X}{q}$ with $|\supp(\Ll)|\ge s+1$. Then fewer
than $|\supp(\Ll)|$ vertices $y\notin X$ can satisfy
$L_X(y)=\Ll$.
\end{corollary}

\begin{proof}
Put $t=|\supp(\Ll)|$. Since $q=r-1\le s<t$, the condition
$|\supp(\Ll)|=t$ implies $e(\Ll)\ge 2$. If there were $t$ vertices
$y\notin X$ with $L_X(y)=\Ll$, Lemma~\ref{lem:clone} would produce a
Berge cycle of length
\[
2t\ge 2s+2\ge k,
\]
a contradiction.
\end{proof}

\section{Asymptotics and Perron-vector stability}

We first verify that the two constructions are admissible.

\begin{lemma}\label{lem:candidates-legal}
The following statements hold.
\begin{itemize}
\item[(i)] $\Ss_{n,s}^{(r)}$ is Berge-$C_{\ge 2s+1}$-free.
\item[(ii)] $\Ss_{n,s}^{(r),+}$ is Berge-$C_{\ge 2s+2}$-free.
\end{itemize}
\end{lemma}
\begin{proof}
Every Berge cycle of length \(\ell\) in an \(r\)-graph projects to an
ordinary cycle of length \(\ell\) in its \(2\)-shadow. Indeed, if $v_1,e_1,v_2,e_2,\ldots,v_\ell,e_\ell$ is a Berge cycle, then $v_iv_{i+1}\in E(\partial_2H)$ for every \(i\), where indices are taken modulo \(\ell\). Since the
base vertices \(v_1,\ldots,v_\ell\) are distinct, they form an
ordinary cycle in \(\partial_2H\). It therefore suffices to bound the
lengths of cycles in the relevant \(2\)-shadows.

First consider \(\mathcal{S}_{n,s}^{(r)}\). Let \(S\) be its
distinguished \(s\)-set and put \(Y:=V\setminus S\). Then,  $\partial_2\mathcal{S}_{n,s}^{(r)}=K_s\vee\overline{K}_{n-s}$.
Let \(C\) be a cycle in this graph, and put $a:=|V(C)\cap S|$ and $b:=|V(C)\cap Y|$. Since \(Y\) is independent, no two vertices of \(V(C)\cap Y\) are
consecutive on \(C\). Hence \(b\le a\), and therefore $|C|=a+b\le2a\le2s$. Thus \(\partial_2\mathcal{S}_{n,s}^{(r)}\) contains no cycle of
length at least \(2s+1\), proving that
\(\mathcal{S}_{n,s}^{(r)}\) is Berge-\(C_{\ge2s+1}\)-free.

Now consider \(\mathcal{S}_{n,s}^{(r),+}\), and let
\(\{u,v\}\subseteq Y\) be its distinguished outside pair. Its
\(2\)-shadow is $\partial_2\mathcal{S}_{n,s}^{(r),+}
   =K_s\vee\left(K_2\cup\overline{K}_{n-s-2}\right)$,
where the copy of \(K_2\) is induced by \(\{u,v\}\).
Let \(C\) be a cycle in this graph, and define \(a\) and \(b\) as
above.

If \(C\) does not use the edge \(uv\), then no two outside vertices
are consecutive on \(C\), so \(b\le a\) and $|C|\le2s$. Suppose that \(C\) uses \(uv\). Then the vertices of
\(V(C)\cap Y\) form \(b-1\) blocks around \(C\): one block is the
two-vertex block \(\{u,v\}\), and every remaining outside vertex
forms a singleton block. Consecutive outside blocks must be separated
by vertices of \(S\). Consequently, we get $b-1\le a$.  It follows that $|C|=a+b\le2a+1\le2s+1$. Therefore the \(2\)-shadow contains no cycle of length at least
\(2s+2\), and hence
\(\mathcal{S}_{n,s}^{(r),+}\) is
Berge-\(C_{\ge2s+2}\)-free.
\end{proof}

\begin{lemma}\label{lem:candidate-lower-bound}
For fixed $r$ and $s$, we have
\[
\rho(\Ss_{n,s}^{(r)})
\ge
\left(L_{r,s}+o(1)\right)n^{q/r}.
\]
Consequently, we also have 
\[
\rho(\Ss_{n,s}^{(r),+})
\ge
\left(L_{r,s}+o(1)\right)n^{q/r}.
\]
\end{lemma}

\begin{proof}
Assign the coordinate $\alpha:=\left(\frac{q}{rs}\right)^{1/r}$ to each vertex of $S$, and the coordinate $\beta:=\left(\frac{1}{r(n-s)}\right)^{1/r}$ to each vertex outside $S$. Then we have $s\alpha^r+(n-s)\beta^r=1$. The edges meeting the outside set in exactly one vertex contribute
\begin{align*}
r(n-s)\binom{s}{q}\alpha^q\beta
=
\binom{s}{q}\left(\frac{q}{s}\right)^{q/r}(n-s)^{q/r}=
\left(L_{r,s}+o(1)\right)n^{q/r}.
\end{align*}
All remaining contributions are nonnegative. The second assertion
follows from
$\Ss_{n,s}^{(r)}\subseteq\Ss_{n,s}^{(r),+}$.
\end{proof}

We also record that the candidate spectral radii strictly increase with
the number of vertices.

\begin{lemma}\label{lem:candidate-monotonicity}
For fixed $r$ and $s$, we have $\rho(\Ss_{m,s}^{(r)})<\rho(\Ss_{m+1,s}^{(r)})$
for $m\ge s+1$, and $\rho(\Ss_{m,s}^{(r),+})<\rho(\Ss_{m+1,s}^{(r),+})$ for $m\ge s+2$.
\end{lemma}

\begin{proof}
We use the standard fact that if $\Hh$ is a connected $r$-uniform
hypergraph and $\Hh'$ is a proper subhypergraph of $\Hh$, then $\rho(\Hh')<\rho(\Hh)$; see, for example, \cite[Lemma~2]{LuMan2016}.

For the first assertion, identify $\Ss_{m,s}^{(r)}$ with the
subhypergraph of $\Ss_{m+1,s}^{(r)}$ obtained by deleting one vertex
outside the distinguished set $S$. Since $s\ge r$, the hypergraph
$\Ss_{m+1,s}^{(r)}$ is connected, and the inclusion is proper.
Hence, $\rho(\Ss_{m,s}^{(r)})<\rho(\Ss_{m+1,s}^{(r)})$.

For the second assertion, choose the distinguished outside pair
$\{u,v\}$ to be the same in the two constructions and take the new
vertex outside $\{u,v\}$. Then $\Ss_{m,s}^{(r),+}$ is a proper
subhypergraph of $\Ss_{m+1,s}^{(r),+}$. The latter is connected since
it contains $\Ss_{m+1,s}^{(r)}$ as a spanning connected subhypergraph.
Therefore, we get $\rho(\Ss_{m,s}^{(r),+})<\rho(\Ss_{m+1,s}^{(r),+})$, as required.
\end{proof}

We now determine the leading asymptotic term.

\begin{theorem}\label{thm:asymptotic}
For fixed $r\ge 3$ and $k\ge 2r+1$, we have 
\begin{equation}\label{eq:asymptotic-main}
\spex_r(n,k)
=
\left(L_{r,s}+o(1)\right)n^{q/r}.
\end{equation}
\end{theorem}

\begin{proof}
The lower bound follows from Lemmas~\ref{lem:candidates-legal} and
\ref{lem:candidate-lower-bound}. We prove the upper bound.

Let $\Hh$ be an $n$-vertex Berge-$C_{\ge k}$-free $r$-graph, and let
$x_1\ge x_2\ge\cdots\ge x_n\ge 0$ with $\sum_{i=1}^n x_i^r=1$ be a normalized vector attaining $\rho(\Hh)$. Fix an integer $M\ge s$
and put $X=[M]$ and $Y=[n]\setminus X$. Partition $E(\Hh)$ into $E_0,E_1,E_2$ according as an edge contains
zero, exactly one, or at least two vertices of $Y$. Then we have
\begin{equation}\label{eq:rayleigh-split}
\frac{\rho(\Hh)}{r}
=
\sum_{e\in E_0}x_e
+
\sum_{e\in E_1}x_e
+
\sum_{e\in E_2}x_e.
\end{equation}

Since $E_0\subseteq\binom{X}{r}$, we immediately get 
\begin{equation}\label{eq:E0-bound}
\sum_{e\in E_0}x_e=O_M(1).
\end{equation}

To estimate $E_2$, let $\phi$ be the map from
Lemma~\ref{lem:orientation}. At most $\kappa M$ edges of $E_2$ are
assigned by $\phi$ to vertices of $X$, so their total weight is
$O_M(1)$. If $e\in E_2$ is assigned to $y\in Y$, choose
$z\in(e\cap Y)\setminus\{y\}$. Since
\[
Mx_{M+1}^r\le \sum_{i=1}^M x_i^r\le 1,
\]
we have $x_z\le x_{M+1}\le M^{-1/r}$, and hence $x_e\le M^{-1/r}x_y$. Using the load bound in Lemma~\ref{lem:orientation}, we obtain
\begin{equation}\label{eq:E2-bound}
\sum_{e\in E_2}x_e
\le
\kappa M^{-1/r}\sum_{y\in Y}x_y+O_M(1).
\end{equation}

It remains to estimate $E_1$. For $y\in Y$, put
$\Ll_y:=L_X(y)\subseteq\binom{X}{q}$.
Then we have
\begin{equation}\label{eq:E1-link-sum}
\sum_{e\in E_1}x_e
=
\sum_{y\in Y}x_y\sum_{F\in\Ll_y}x_F.
\end{equation}
Let $Y_{\mathrm{large}}
:=
\{y\in Y:|\supp(\Ll_y)|\ge s+1\}$. For each fixed $q$-graph $\Ll\subseteq\binom{X}{q}$ with support of
size at least $s+1$, Corollary~\ref{cor:link-type} shows that fewer
than $|\supp(\Ll)|\le M$ vertices $y$ satisfy $\Ll_y=\Ll$. Since
there are at most $2^{\binom{M}{q}}$ possible link hypergraphs, we obtain that 
\[
|Y_{\mathrm{large}}|
\le
M2^{\binom{M}{q}}
=
O_M(1).
\]
Their total contribution to \eqref{eq:E1-link-sum} is therefore
$O_M(1)$.

For $y\in Y\setminus Y_{\mathrm{large}}$, the support of $\Ll_y$ has
size at most $s$. Since the coordinates are nonincreasing, we get
\[
\sum_{F\in\Ll_y}x_F
\le
e_q(x_1,\ldots,x_s).
\]
Consequently, from \eqref{eq:E1-link-sum} we obtain
\begin{equation}\label{eq:E1-bound}
\sum_{e\in E_1}x_e
\le
e_q(x_1,\ldots,x_s)\sum_{y\in Y}x_y+O_M(1).
\end{equation}

Put $T:=\sum_{i=1}^s x_i^r$. By Lemma~\ref{lem:elementary-symmetric}, we get
\[
e_q(x_1,\ldots,x_s)
\le
\binom{s}{q}\left(\frac{T}{s}\right)^{q/r}.
\]
Moreover, H\"older's inequality gives $\sum_{y\in Y}x_y\le n^{q/r}(1-T)^{1/r}$.
Combining \eqref{eq:rayleigh-split}, \eqref{eq:E0-bound},
\eqref{eq:E2-bound}, and \eqref{eq:E1-bound}, we obtain, for fixed
$M$, that
\begin{align}
\frac{\rho(\Hh)}{n^{q/r}}
&\le
r\binom{s}{q}s^{-q/r}T^{q/r}(1-T)^{1/r}
+r\kappa M^{-1/r}
+o_n(1).
\label{eq:asymptotic-upper-fixed-M}
\end{align}
The function $f(T):=T^{q/r}(1-T)^{1/r}$ has a unique maximum on $[0,1]$ at $T=q/r$, and
$\max_{0\le T\le 1}f(T)=\frac{q^{q/r}}{r}$.
Thus \eqref{eq:asymptotic-upper-fixed-M} gives
\[
\limsup_{n\to\infty}
\frac{\spex_r(n,k)}{n^{q/r}}
\le
L_{r,s}+r\kappa M^{-1/r},
\]
for arbitrary $M\ge s$. Hence, we get
\[\limsup_{n\to\infty}
\frac{\spex_r(n,k)}{n^{q/r}}\le \inf_{M\ge s}L_{r,s}+r\kappa M^{-1/r}=L_{r,s},\]
which proves the required upper bound.
\end{proof}

We next establish a quantitative stability theorem. 

\begin{theorem}\label{thm:stability}
For every $\varepsilon>0$, there exist
$\delta=\delta(r,k,\varepsilon)>0$ and
$n_0=n_0(r,k,\varepsilon)$ with the following property. Let $n\ge n_0$,
let $\Hh$ be an $n$-vertex Berge-$C_{\ge k}$-free $r$-graph, and let
$x_1\ge\cdots\ge x_n\ge 0$ with $\sum_{i=1}^n x_i^r=1$ attain $\rho(\Hh)$. If $\rho(\Hh)\ge (L_{r,s}-\delta)n^{q/r}$,
then, with $S=[s]$, the following statements hold:
\begin{itemize}
\item[(i)] $x_{s+1}<\varepsilon$;
\item[(ii)] $\left|x_i^r-\frac{q}{rs}\right|<\varepsilon$ for $1\le i\le s$;
\item[(iii)] $\left|\sum_{i>s}x_i^r-\frac1r\right|<\varepsilon$
and $\left|
\frac{\sum_{i>s}x_i}{n^{q/r}}-r^{-1/r}
\right|<\varepsilon$;
\item[(iv)] fewer than $\varepsilon n$ vertices
$y\in V(\Hh)\setminus S$ satisfy $L_S(y)\ne\binom{S}{q}$.
\end{itemize}
\end{theorem}

\begin{proof}
Suppose the theorem is false. Then there is an $\varepsilon_0>0$ for
which no such pair $(\delta,n_0)$ exists. For each positive integer
$j$, choose a Berge-$C_{\ge k}$-free $r$-graph $\Hh_j$ on
$n_j\ge j$ vertices, together with a normalized maximizing vector
$x^{(j)}$, such that $\rho(\Hh_j)\ge (L_{r,s}-j^{-1})n_j^{q/r}$,
but at least one of conclusions (i)--(iv), with $\varepsilon_0$ in
place of $\varepsilon$, fails. After passing to a subsequence and
relabeling $n_j$ as $n$, Theorem~\ref{thm:asymptotic} gives
\begin{equation}\label{eq:stability-near-extremal}
\frac{\rho(\Hh_n)}{n^{q/r}}\longrightarrow L_{r,s}.
\end{equation}

For readability, suppress the superscript $(n)$ and write $T_n:=\sum_{i=1}^s x_i^r$ and $\eta_n:=x_{s+1}$. We first prove that
\begin{equation}\label{eq:eta-to-zero}
\eta_n\longrightarrow0.
\end{equation}
Suppose otherwise. After passing to a subsequence, there exists a
constant \(\eta>0\) such that $\eta_n=x_{s+1}\ge\eta$ for every \(n\) in the subsequence. Fix an integer \(M\ge s+1\), and put $X=[M]$ and  $Y=[n]\setminus[M]$.
Since \(x_{s+1}\ge\eta\), the total \(r\)-mass on \(Y\) satisfies $\sum_{y\in Y}x_y^r\le 1-T_n-\eta^r$. Repeating the edge decomposition used in the proof of
Theorem~\ref{thm:asymptotic}, with \(M\) fixed, gives
\[
\frac{\rho(H_n)}{n^{q/r}}
\le
r\binom{s}{q}s^{-q/r}
T_n^{q/r}
\bigl(1-\eta^r-T_n\bigr)^{1/r}
+
r\kappa M^{-1/r}
+
o(1),
\]
where \(n\to\infty\) while \(M\) is fixed.

For \(A:=1-\eta^r\), we have $\max_{0\le T\le A}T^{q/r}(A-T)^{1/r}=A\frac{q^{q/r}}{r}$. Consequently, we get
\[ 
\limsup_{n\to\infty}
\frac{\rho(H_n)}{n^{q/r}}
\le
L_{r,s}(1-\eta^r)+r\kappa M^{-1/r}.
\]
Choose \(M\) sufficiently large such that $r\kappa M^{-1/r}<\frac12L_{r,s}\eta^r$. Then we get
\[
\limsup_{n\to\infty}
\frac{\rho(H_n)}{n^{q/r}}
\le
L_{r,s}\left(1-\frac{\eta^r}{2}\right)
<
L_{r,s},
\]
contradicting \eqref{eq:stability-near-extremal}. This proves
\eqref{eq:eta-to-zero}.

We now take $X=S$ and $Y=V(\Hh_n)\setminus S$. Split the edges
according to the number of vertices they contain in $Y$. The total
weight of edges contained in $S$ is $O_{r,s}(1)$. For edges containing
at least two vertices of $Y$, use the assignment from
Lemma~\ref{lem:orientation}. At most $\kappa s$ such edges are assigned
to $S$. Every remaining edge is assigned to some $y\in Y$ and contains
another vertex of $Y$, whose coordinate is at most $\eta_n$. Hence, we get
\begin{equation}\label{eq:stability-E2}
\sum_{\substack{e\in E(\Hh_n)\\ |e\cap Y|\ge 2}}x_e
\le
\kappa\eta_n\sum_{y\in Y}x_y+O_{r,s}(1)
=
o(n^{q/r}).
\end{equation}
For an edge containing exactly one vertex $y\in Y$, its core part
belongs to $\binom{S}{q}$. Therefore, we get
\begin{equation}\label{eq:stability-master}
\frac{\rho(\Hh_n)}{r}
\le
e_q(x_1,\ldots,x_s)\sum_{i>s}x_i+o(n^{q/r}).
\end{equation}
Using Lemma~\ref{lem:elementary-symmetric} and H\"older's inequality, we obtain
\begin{equation}\label{eq:stability-one-variable}
\frac{\rho(\Hh_n)}{n^{q/r}}
\le
r\binom{s}{q}s^{-q/r}
T_n^{q/r}(1-T_n)^{1/r}+o(1).
\end{equation}
The right-hand side has a unique maximum $L_{r,s}$, attained at
$T_n=q/r$. In view of \eqref{eq:stability-near-extremal}, we get
\begin{equation}\label{eq:T-limit}
T_n\longrightarrow \frac{q}{r}.
\end{equation}

We next prove that the core coordinates are asymptotically equal. If
not, pass to a subsequence on which
$(x_1,\ldots,x_s)$ converges to a vector
$(a_1,\ldots,a_s)$ that is not constant. By \eqref{eq:T-limit}, we get 
$\sum_i a_i^r=q/r$. Lemma~\ref{lem:elementary-symmetric} is then strict
at $(a_1,\ldots,a_s)$. Together with \eqref{eq:stability-master} and
H\"older's inequality, this would make the limit superior in
\eqref{eq:stability-near-extremal} strictly smaller than $L_{r,s}$, a
contradiction. Hence, we conclude that
\begin{equation}\label{eq:core-coordinate-limit}
x_i^r=\frac{q}{rs}+o(1) \text{ for }1\le i\le s.
\end{equation}
This proves the asymptotic form of conclusion (ii), while
\eqref{eq:T-limit} gives
\begin{equation}\label{eq:tail-r-mass}
\sum_{i>s}x_i^r=\frac1r+o(1).
\end{equation}

Let $A_n:=e_q(x_1,\ldots,x_s)$ and $R_n:=\sum_{i>s}x_i$. By \eqref{eq:core-coordinate-limit}, we get
\[
A_n
=
\binom{s}{q}\left(\frac{q}{rs}\right)^{q/r}+o(1).
\]
The left-hand side of \eqref{eq:stability-master} is asymptotic to
$(L_{r,s}/r)n^{q/r}$, so $R_n\ge \left(r^{-1/r}-o(1)\right)n^{q/r}$. On the other hand, H\"older's inequality and
\eqref{eq:tail-r-mass} give the matching upper bound. Thus
\begin{equation}\label{eq:tail-l1-mass}
R_n
=
\left(r^{-1/r}+o(1)\right)n^{q/r}.
\end{equation}

It remains to prove the link conclusion. For $y\notin S$, put $W_y:=\sum_{F\in L_S(y)}x_F$. Since $W_y\le A_n$, the contribution of edges with exactly one outside
vertex is $\sum_{y\notin S}x_y W_y$. By
\eqref{eq:stability-E2}, \eqref{eq:stability-near-extremal},
\eqref{eq:core-coordinate-limit}, and \eqref{eq:tail-l1-mass}, we have
\begin{equation}\label{eq:link-deficit}
\sum_{y\notin S}x_y(A_n-W_y)=o(n^{q/r}).
\end{equation}
Let $B_n:=\left\{
y\notin S:L_S(y)\ne\binom{S}{q}
\right\}$.
By \eqref{eq:core-coordinate-limit}, there exists a constant
$c=c(r,s)>0$ such that $\min_{F\in\binom{S}{q}}x_F\ge c$ for all sufficiently large $n$. Every $y\in B_n$ misses at least one core $q$-set, and hence
\eqref{eq:link-deficit} yields
\begin{equation}\label{eq:bad-l1-small}
\sum_{y\in B_n}x_y=o(n^{q/r}).
\end{equation}
If $|B_n|\ge \varepsilon_0n$ along a subsequence, then the complement
of $B_n$ outside $S$ has at most $(1-\varepsilon_0+o(1))n$ vertices.
By H\"older's inequality and \eqref{eq:tail-r-mass}, we get
\[
\sum_{\substack{y\notin S\\ y\notin B_n}}x_y
\le
(1-\varepsilon_0+o(1))^{q/r}r^{-1/r}n^{q/r}.
\]
Together with \eqref{eq:bad-l1-small}, this contradicts
\eqref{eq:tail-l1-mass}. Therefore $|B_n|=o(n)$.

We have shown that every sequence satisfying
\eqref{eq:stability-near-extremal} satisfies the asymptotic forms of
(i)--(iv), which contradicts the choice of the sequence. The theorem
follows.
\end{proof}

\section{Proof of the exact theorem}

We now upgrade the stability conclusion to the exact extremal
structure. We first isolate the cycle constructions needed in both
parity cases.

Let \(S\) be an \(s\)-set and put
\[
G:=
\left\{
g\in V(H)\setminus S:
L_S(g)=\binom{S}{q}
\right\}.
\]
We call the vertices of \(G\) \emph{good}. We shall repeatedly use the following realization observation. Let
\(g\in G\), and let \(a,b\in S\) be prescribed. Since
\(s\ge r=q+1\), there exist distinct sets $F,F'\in\binom{S}{q}$ such that \(a\in F\) and \(b\in F'\). Hence the pairs \(ga\) and
\(gb\) can be represented by the distinct hyperedges $\{g\}\cup F$ and $\{g\}\cup F'$. Every such core-link edge contains exactly one vertex outside \(S\).
Consequently, core-link edges associated with different good vertices
are automatically distinct. Moreover, every hyperedge used below to
represent a pair of vertices outside \(S\) contains at least two
vertices outside \(S\), and therefore cannot coincide with a
core-link edge. Thus, whenever the exceptional hyperedges explicitly
chosen below are distinct, all hyperedges representing the resulting
Berge cycle are pairwise distinct.

\begin{lemma}\label{lem:cycle-forcing}
Let \(H\) be Berge-\(C_{\ge k}\)-free, let \(S\) be an \(s\)-set,
and assume that $|G|\ge s+2r$. Then the following statements hold.

\begin{itemize}
\item[\rm (i)]
If \(k=2s+1\), then every edge of \(H\) contains at most one
vertex of \(G\).

\item[\rm (ii)]
Put $B:=V(H)\setminus(S\cup G)$, and form a bipartite graph \(\Gamma\) with parts \(B\) and \(G\),
where \(bg\in E(\Gamma)\) if there exists an edge \(e\in E(H)\)
such that $\{b,g\}\subseteq e$ and $|e\setminus S|\ge2$.
Then, for every \(b\in B\), we have
\[
d_\Gamma(b)\le
\begin{cases}
1,   & k=2s+1,\\
r-1, & k=2s+2.
\end{cases}
\]

\item[\rm (iii)]
Suppose that \(k=2s+2\), and define
\[
D_G:=
\left\{
e\in E(H):
|e\setminus S|\ge2,\quad e\setminus S\subseteq G
\right\}.
\]
If \(|D_G|\ge2\), then there exists a fixed pair
\(P_0\in\binom{G}{2}\) such that $e\setminus S=P_0$ for every \(e\in D_G\).
\end{itemize}
\end{lemma}

\begin{proof}
For (i), suppose that an edge \(e\) contains distinct vertices
\(u,v\in G\). Write $S=\{x_1,\ldots,x_s\}$,
and choose distinct vertices $g_1,\ldots,g_{s-1}\in G\setminus e$. Consider the cyclic base sequence
\[
u,v,x_1,g_1,x_2,g_2,\ldots,
x_{s-1},g_{s-1},x_s.
\]
Represent the pair \(uv\) by \(e\), and represent every remaining
consecutive pair through complete core links. By the realization
observation preceding the lemma, all representing hyperedges can be
chosen pairwise distinct. We obtain a Berge cycle of length \(2s+1\),
a contradiction.

For (ii), first suppose that \(k=2s+1\). If some \(b\in B\) has two
distinct good neighbors \(g,h\), choose hyperedges \(e,f\in E(H)\)
such that $\{b,g\}\subseteq e$ and $\{b,h\}\subseteq f$. By part (i), we must have \(e\ne f\). Choose distinct vertices
$g_1,\ldots,g_{s-1}\in G\setminus\{g,h\}$. The cyclic base sequence
\[
b,g,x_1,g_1,x_2,g_2,\ldots,
x_{s-1},g_{s-1},x_s,h
\]
can be represented by \(e\) on \(bg\), by \(f\) on \(hb\), and by
complete core links on all other consecutive pairs. All representing
hyperedges are pairwise distinct, so this is a Berge cycle of length
\(2s+2\ge k\), a contradiction. Hence \(d_\Gamma(b)\le1\).

Now suppose that \(k=2s+2\). If \(d_\Gamma(b)\ge r\), then the \(r\)
distinct good neighbors of \(b\) cannot all lie with \(b\) in one
\(r\)-edge. Consequently, there exist distinct hyperedges \(e,f\)
and distinct good vertices \(g,h\) such that $\{b,g\}\subseteq e$ and $\{b,h\}\subseteq f$. Using the same cyclic base sequence as above gives a Berge cycle of
length \(2s+2=k\), again a contradiction. Thus
\(d_\Gamma(b)\le r-1\).

For (iii), take distinct edges \(e,f\in D_G\). We claim that every
pair $P\in\binom{e\setminus S}{2}$ is equal to every pair $Q\in\binom{f\setminus S}{2}$. Suppose otherwise. First assume that \(P\) and \(Q\) share one vertex. Write $P=\{a,b\}$ and $Q=\{b,c\}$, where \(a,b,c\) are distinct. Choose distinct vertices $g_1,\ldots,g_{s-1}\in G\setminus\{a,b,c\}$.
Use the cyclic base sequence
\[
a,b,c,x_1,g_1,x_2,g_2,\ldots,
x_{s-1},g_{s-1},x_s.
\]
Represent \(ab\) by \(e\), represent \(bc\) by \(f\), and represent
all other consecutive pairs through complete core links. This gives
a Berge cycle of length \(2s+2\), a contradiction.

Next assume that \(P\) and \(Q\) are disjoint. Write $P=\{a,b\}$ and  $Q=\{c,d\}$. Choose distinct vertices $g_1,\ldots,g_{s-2}
   \in G\setminus\{a,b,c,d\}$. The cyclic base sequence
\[
a,b,x_1,g_1,x_2,g_2,\ldots,
x_{s-2},g_{s-2},x_{s-1},c,d,x_s
\]
is represented by \(e\) on \(ab\), by \(f\) on \(cd\), and by
complete core links on all remaining consecutive pairs. Again this
is a Berge cycle of length \(2s+2\), a contradiction.

It follows that every pair in
\(\binom{e\setminus S}{2}\) equals every pair in
\(\binom{f\setminus S}{2}\). Since both \(e\setminus S\) and
\(f\setminus S\) contain at least two vertices, each of them must
therefore consist of exactly the same two vertices. As \(e\) and
\(f\) were arbitrary distinct members of \(D_G\), there is a fixed
pair \(P_0\in\binom{G}{2}\) such that $e\setminus S=P_0$ for every \(e\in D_G\).
\end{proof}

We can now prove Theorem~\ref{thm:main}.

\begin{proof}[Proof of Theorem~\ref{thm:main}]
Suppose, to the contrary, that the theorem fails for infinitely many
values of $n$. For each such $n$, choose a spectral extremal
$n$-vertex Berge-$C_{\ge k}$-free $r$-graph $\Hh_n$ that is not
isomorphic to the asserted candidate. By Lemma~\ref{lem:components},
$\Hh_n$ has a connected component $\Hh_n^\star$ satisfying $\rho(\Hh_n^\star)=\rho(\Hh_n)$. Let
$m_n:=|V(\Hh_n^\star)|$. The candidate lower bound gives
\[
\rho(\Hh_n^\star)=\rho(\Hh_n)
\ge \left(L_{r,s}-o(1)\right)n^{q/r}.
\]
In particular, $\rho(\Hh_n^\star)\to\infty$, and the trivial bound
$\rho(\Hh_n^\star)\le r\binom{m_n}{r}$ shows that $m_n\to\infty$.
We may therefore apply Theorem~\ref{thm:asymptotic} to the sequence of
components, obtaining
\[
\left(L_{r,s}-o(1)\right)n^{q/r}
\le
\rho(\Hh_n^\star)
\le
\left(L_{r,s}+o(1)\right)m_n^{q/r}.
\]
Hence, we get
\begin{equation}\label{eq:component-size}
\frac{m_n}{n}\longrightarrow 1
\end{equation}
and
\begin{equation}\label{eq:component-near-extremal}
\frac{\rho(\Hh_n^\star)}{m_n^{q/r}}
\longrightarrow L_{r,s}.
\end{equation}

Let $x$ be the positive normalized Perron vector of $\Hh_n^\star$,
with vertices ordered by nonincreasing coordinate, and put $S=[s]$.
Applying Theorem~\ref{thm:stability} along the sequence
\eqref{eq:component-near-extremal}, we obtain
\begin{equation}\label{eq:exact-stability-data}
\eta_n:=\max_{y\notin S}x_y=o(1),
\qquad
|B|=o(m_n),
\qquad
c_n:=\min_{F\in\binom{S}{q}}x_F\ge c_0>0,
\end{equation}
where $G:=\left\{
y\notin S:L_S(y)=\binom{S}{q}
\right\}$ and $B:=V(\Hh_n^\star)\setminus(S\cup G)$. In particular, $|G|=m_n-o(m_n)$, so the hypotheses of
Lemma~\ref{lem:cycle-forcing} hold for all sufficiently large $n$.

Let $\Tt:=\Ss_{m_n,s}^{(r)}$ be the base construction on the same vertex set and with core $S$.
Every edge in
\[
\Dd:=E(\Hh_n^\star)\setminus E(\Tt)
\]
contains at least two vertices outside $S$. Split $\Dd_B:=\{e\in\Dd:e\cap B\ne\varnothing\}$ and $\Dd_G:=\Dd\setminus\Dd_B$. If $k=2s+1$, Lemma~\ref{lem:cycle-forcing}(i) gives
$\Dd_G=\varnothing$.

We first control the total Perron weight of $\Dd_B$. Define
\[
d:=
\begin{cases}
1,&k=2s+1,\\
r-1,&k=2s+2.
\end{cases}
\]
Let $\Dd'\subseteq\Dd_B$ be nonempty, and put $U:=B\cap\bigcup_{e\in\Dd'}e$ and $W:=G\cap\bigcup_{e\in\Dd'}e$. Every vertex of $W$ has a neighbor in $U$ in the bipartite graph of
Lemma~\ref{lem:cycle-forcing}(ii), while every vertex of $U$ has degree
at most $d$. Hence, we get $|W|\le d|U|$. Every edge of $\Dd'$ is contained in $S\cup U\cup W$, so
Lemma~\ref{lem:sparsity} gives
\[
|\Dd'|
\le
\kappa\bigl(s+(d+1)|U|\bigr)
\le
K|U|,\]
where $K:=\kappa(s+d+1)$. Lemma~\ref{lem:bounded-assignment} therefore yields a map
$\psi:\Dd_B\to B$ such that $\psi(e)\in e$ and
$|\psi^{-1}(b)|\le K$ for every $b\in B$. If $\psi(e)=b$, then $e$
contains another vertex outside $S$, whose coordinate is at most
$\eta_n$. Therefore, we get
\begin{equation}\label{eq:bad-edge-weight}
\sum_{e\in\Dd_B}x_e
\le
K\eta_n\sum_{b\in B}x_b.
\end{equation}

Every bad vertex $b\in B$ misses at least one base edge
$\{b\}\cup F_b$ with $F_b\in\binom{S}{q}$. These missing edges are
distinct for different $b$, and hence
\begin{equation}\label{eq:missing-base-weight}
\sum_{e\in E(\Tt)\setminus E(\Hh_n^\star)}x_e
\ge
c_n\sum_{b\in B}x_b.
\end{equation}
We now account for the parity-dependent all-good exceptional edges. If
$k=2s+1$, set $W^+:=0$ and $\Tt^+:=\Tt$.
As noted above, $\Dd_G=\varnothing$. Suppose instead that $k=2s+2$. Choose a pair
$P=\{u,v\}\in\binom{G}{2}$ maximizing $x_u x_v$, and define
\[
\Tt^+
:=
\Tt\cup
\left\{
P\cup R:R\in\binom{S}{r-2}
\right\},
\]
and
\[
W^+
:=
x_u x_v\,e_{r-2}(x_i:i\in S).
\]
We claim that
\begin{equation}\label{eq:good-edge-comparison}
\sum_{e\in\Dd_G}x_e\le W^+.
\end{equation}
If $|\Dd_G|\ge 2$, Lemma~\ref{lem:cycle-forcing}(iii) gives a fixed
pair $P_0\in\binom{G}{2}$ such that every edge of $\Dd_G$ has the form
$P_0\cup R$ with $R\in\binom{S}{r-2}$. Hence, we get
\[
\sum_{e\in\Dd_G}x_e
\le
x_{P_0}e_{r-2}(x_i:i\in S)
\le
W^+.
\]
If $\Dd_G=\{e\}$ and $|e\setminus S|=2$, writing
$Q=e\setminus S$ gives
\[
x_e
<
x_Q e_{r-2}(x_i:i\in S)
\le
W^+,
\]
because $\binom{s}{r-2}>1$ and all core coordinates are positive. If
$\Dd_G=\{e\}$ and $t:=|e\setminus S|\ge 3$, choose
$Q\in\binom{e\setminus S}{2}$. Then we have $x_e\le x_Q\eta_n^{t-2}$,
whereas $W^+\ge x_Q e_{r-2}(x_i:i\in S)$. The elementary symmetric polynomial on the right is bounded away from
zero by \eqref{eq:exact-stability-data}, while $\eta_n=o(1)$, so the
inequality is strict for large $n$. The case $\Dd_G=\varnothing$ is
trivial. Moreover, equality in \eqref{eq:good-edge-comparison} can hold
only when $\Dd_G$ is the complete special family on a fixed pair
$P_0$ with maximal product $x_{P_0}$.

The difference of the two Rayleigh sums is
\begin{align}
\frac{P_{\Tt^+}(x)-P_{\Hh_n^\star}(x)}{r}
&=
\sum_{e\in E(\Tt)\setminus E(\Hh_n^\star)}x_e
+W^+
-
\sum_{e\in\Dd_B}x_e
-
\sum_{e\in\Dd_G}x_e.
\label{eq:unified-rayleigh-difference}
\end{align}
By \eqref{eq:bad-edge-weight}, \eqref{eq:missing-base-weight}, and
\eqref{eq:good-edge-comparison}, if $B\ne\varnothing$, then for all
large $n$, we have
\[
P_{\Tt^+}(x)-P_{\Hh_n^\star}(x)
\ge
r(c_n-K\eta_n)\sum_{b\in B}x_b
>
0.
\]
The hypergraph $\Tt^+$ is admissible by Lemma
\ref{lem:candidates-legal}; after adjoining $n-m_n$ isolated vertices,
this contradicts the extremality of $\Hh_n$. Hence
\begin{equation}\label{eq:B-empty}
B=\varnothing.
\end{equation}

All outside vertices are now good, so every base edge meeting the
outside set is present. If $k=2s+1$, Lemma
\ref{lem:cycle-forcing}(i) gives $\Dd=\varnothing$. Thus
$\Hh_n^\star\subseteq\Tt$. If a core edge were missing, positivity of
the Perron vector $x$ would give $P_{\Tt}(x)>P_{\Hh_n^\star}(x)=\rho(\Hh_n^\star)$,
contradicting extremality. Therefore, $\Hh_n^\star\cong\Ss_{m_n,s}^{(r)}$.

If $k=2s+2$, then \eqref{eq:unified-rayleigh-difference}, with
$B=\varnothing$, can be nonpositive only if no core edge is missing and
equality holds in \eqref{eq:good-edge-comparison}. The equality
characterization above implies $\Hh_n^\star\cong\Ss_{m_n,s}^{(r),+}$.

Finally, if $m_n<n$, Lemma~\ref{lem:candidate-monotonicity} gives $\rho(\Hh_n)=\rho(\Hh_n^\star)<\rho(\Ss_{n,s}^{(r)})$
in the odd case, or $\rho(\Hh_n)=\rho(\Hh_n^\star)<\rho(\Ss_{n,s}^{(r),+})$ in the even case. Both contradict extremality. Hence $m_n=n$, and
$\Hh_n$ is the asserted candidate, contrary to its choice. This proves
Theorem~\ref{thm:main}.
\end{proof}

Corollary~\ref{cor:asymptotic-intro} follows immediately from
Theorems~\ref{thm:main} and \ref{thm:asymptotic}.

\section{Concluding remarks}

The proof uses only the linear edge bound of F\"uredi, Kostochka and
Luo \cite{FKL1}, through the hereditary sparsity estimate in
Lemma~\ref{lem:sparsity}; the exact all-$n$ edge theorem
\cite{FKL2} is not needed. Starting from this sparse structural input,
the proof separates the large and small Perron coordinates, shows that
edges containing several small coordinates are negligible, and reduces
the leading contribution to complete links over a fixed core. The
Perron-vector stability theorem then converts asymptotic optimality into
an almost-complete core-link structure, while the final weighted
comparison makes the description exact.

The result exhibits a concentration phenomenon absent from the
edge-extremal problem. Edge extremality is governed by a tree-like
assembly of complete $r$-uniform blocks of order $k-1$. Spectral
extremality instead places asymptotically all large coordinates on the
fixed $s$-vertex core, with almost every remaining vertex having the
complete $q$-link to that core. In the even case, exactly one additional
outside pair is permitted.

The assumption $k\ge 2r+1$ enters through $s\ge r$. It is used in the
link-realization steps that construct forbidden Berge cycles and in the
comparison that forces the exact outside structure. The adjacent cases
$k=2r-1$ and $k=2r$, where $s=r-1$, require a different analysis: the
candidate becomes a tight sunflower-type construction, and the balance
between core-supported edges and edges containing several outside
vertices changes. Determining the exact spectral extremal hypergraphs in
this remaining range is a natural open problem.

\section*{Acknowledgments}

Lihua Feng is supported by the National Natural Science Foundation of
China (Nos. 12271527 and 12471022). Lu Lu is supported by the National
Natural Science Foundation of China (No. 12371362). Tingzeng Wu is supported by Natural Science Foundation of Qinghai Province (No. 2025-ZJ-902T), and National Natural Science Foundation of China (No. 12261071).


{\footnotesize
\begin{thebibliography}{99}

\bibitem{BaiLu2018}
S.~Bai and L.~Lu,
\newblock A bound on the spectral radius of hypergraphs with $e$ edges,
\newblock \emph{Linear Algebra Appl.} \textbf{549} (2018), 203--218.

\bibitem{BrooksLinzLuo2026}
G.~Brooks, W.~Linz and R.~Luo,
\newblock Spectral radius and Hamiltonicity of uniform hypergraphs,
\newblock \emph{Discrete Math.} \textbf{349} (2026), 115184.

\bibitem{ChangPearsonZhang2008}
K.-C.~Chang, K.~Pearson and T.~Zhang,
\newblock Perron--Frobenius theorem for nonnegative tensors,
\newblock \emph{Commun. Math. Sci.} \textbf{6} (2008), no.~2, 507--520.

\bibitem{CooperDutle2012}
J.~Cooper and A.~Dutle,
\newblock Spectra of uniform hypergraphs,
\newblock \emph{Linear Algebra Appl.} \textbf{436} (2012), no.~9,
3268--3292.

\bibitem{ErdosGallai1959}
P.~Erd\H{o}s and T.~Gallai,
\newblock On maximal paths and circuits of graphs,
\newblock \emph{Acta Math. Acad. Sci. Hungar.} \textbf{10} (1959),
337--356.

\bibitem{ErgemlidzeEtAl2020}
B.~Ergemlidze, E.~Gy\H{o}ri, A.~Methuku, N.~Salia, C.~Tompkins and
O.~Zamora,
\newblock Avoiding long Berge cycles: the missing cases $k=r+1$ and
$k=r+2$,
\newblock \emph{Combin. Probab. Comput.} \textbf{29} (2020), no.~3,
423--435.

\bibitem{FangChangXuGaoHou2025}
L.~Fang, A.~Chang, W.~Xu, G.~Gao and Y.~Hou,
\newblock The spectral radius of $3$-graphs without Berge paths of given
length,
\newblock \emph{Discrete Appl. Math.} \textbf{377} (2025), 414--428.

\bibitem{FKL1}
Z.~F\"uredi, A.~Kostochka and R.~Luo,
\newblock Avoiding long Berge cycles,
\newblock \emph{J. Combin. Theory Ser. B} \textbf{137} (2019), 55--64.

\bibitem{FKL2}
Z.~F\"uredi, A.~Kostochka and R.~Luo,
\newblock Avoiding long Berge cycles II, exact bounds for all $n$,
\newblock \emph{J. Combin.} \textbf{12} (2021), no.~2, 247--268.

\bibitem{GaoHou}
J.~Gao and X.~Hou,
\newblock The spectral radius of graphs without long cycles,
\newblock\emph{Linear Algebra Appl.} \textbf{566} (2019), 17--33.

\bibitem{GyoriKatonaLemons2016}
E.~Gy\H{o}ri, G.~Y.~Katona and N.~Lemons,
\newblock Hypergraph extensions of the Erd\H{o}s--Gallai theorem,
\newblock \emph{European J. Combin.} \textbf{58} (2016), 238--246.

\bibitem{KangLiuLuWang2021}
L.~Kang, L.~Liu, L.~Lu and Z.~Wang,
\newblock The extremal $p$-spectral radius of Berge hypergraphs,
\newblock \emph{Linear Algebra Appl.} \textbf{610} (2021), 608--624.

\bibitem{KangNikiforovYuan2015}
L.~Kang, V.~Nikiforov and X.~Yuan,
\newblock The $p$-spectral radius of $k$-partite and $k$-chromatic uniform
hypergraphs,
\newblock \emph{Linear Algebra Appl.} \textbf{478} (2015), 81--107.

\bibitem{KeevashLenzMubayi2014}
P.~Keevash, J.~Lenz and D.~Mubayi,
\newblock Spectral extremal problems for hypergraphs,
\newblock \emph{SIAM J. Discrete Math.} \textbf{28} (2014), no.~4,
1838--1854.

\bibitem{Lim2005}
L.~Lim,
\newblock Singular values and eigenvalues of tensors: a variational
approach,
\newblock in \emph{Proc. 1st IEEE Int. Workshop on Computational Advances
in Multi-Sensor Adaptive Processing (CAMSAP 2005)}, 2005, pp.~129--132.

\bibitem{LiuLu2022}
L.~Liu and L.~Lu,
\newblock The $\alpha$-normal labelling method for computing the
$p$-spectral radii of uniform hypergraphs,
\newblock \emph{Linear Multilinear Algebra} \textbf{70} (2022), no.~9,
1648--1672.

\bibitem{LuMan2016}
L.~Lu and S.~Man,
\newblock Connected hypergraphs with small spectral radius,
\newblock \emph{Linear Algebra Appl.} \textbf{509} (2016), 206--227.

\bibitem{NiLiuKang2024}
Z.~Ni, L.~Liu and L.~Kang,
\newblock Spectral Tur\'an-type problems on cancellative hypergraphs,
\newblock \emph{Electron. J. Combin.} \textbf{31} (2024), no.~2, Paper
P2.32.

\bibitem{Nikiforov2014}
V.~Nikiforov,
\newblock Analytic methods for uniform hypergraphs,
\newblock \emph{Linear Algebra Appl.} \textbf{457} (2014), 455--535.

\bibitem{Qi2005}
L.~Qi,
\newblock Eigenvalues of a real supersymmetric tensor,
\newblock \emph{J. Symbolic Comput.} \textbf{40} (2005), no.~6,
1302--1324.

\bibitem{SheFanKang2025}
C.-M.~She, Y.-Z.~Fan and L.~Kang,
\newblock Spectral bipartite Tur\'an problems on linear hypergraphs,
\newblock \emph{Discrete Math.} \textbf{348} (2025), no.~6, 114435.

\bibitem{WangHuangShi2026}
W.-H.~Wang, Z.-Q.~Huang and H.-L.~Shi,
\newblock The spectral radius of the hypergraphs without the expansion of
$P_3$ or Berge-$P_k$,
\newblock \emph{Discrete Appl. Math.} \textbf{380} (2026), 588--598.

\bibitem{WangSunBu}
Y.~Wang, L.~Sun and C.~Bu, 
\newblock The high order spectral radius of graphs without long cycles or paths,
\newblock \emph{arXiv: 2510.04461}.

\bibitem{ZhengLiFan2026}
J.~Zheng, H.~Li and Y.-Z.~Fan,
\newblock Spectral Tur\'an problems for nondegenerate hypergraphs,
\newblock \emph{Electron. J. Combin.} \textbf{33} (2026), no.~1, Paper
P1.42.

\bibitem{ZhouKangLiuShan2020}
Y.~Zhou, L.~Kang, L.~Liu and E.~Shan,
\newblock Extremal problems for the $p$-spectral radius of Berge
hypergraphs,
\newblock \emph{Linear Algebra Appl.} \textbf{600} (2020), 22--39.

\end{thebibliography}
}
\end{document}